\documentclass[11pt, a4paper]{amsart}
\usepackage{fullpage}
\usepackage{amssymb}
\usepackage{commath}
\usepackage{empheq}
\usepackage{amsmath, amsthm}
\usepackage[foot]{amsaddr}
\usepackage{dsfont}
\usepackage{multicol}

\usepackage[parfill]{parskip}
\mathchardef\mhyphen="2D

\usepackage{tikz}
\usetikzlibrary{matrix,arrows,decorations.pathmorphing,decorations.markings}
\usepackage{tikz-cd}

\usepackage{capt-of}

\usepackage{enumitem}

\usepackage{mathrsfs}

\usepackage{url}

\newcommand{\fas}{\mathcal{F}(as)}
\newcommand{\ifas}{\mathcal{IF}(as)}
\newcommand{\llb}{\left\lbrace}
\newcommand{\rrb}{\right\rbrace}

\renewcommand{\phi}{\varphi}

\newcommand{\llangle}{\left\langle}
\newcommand{\rrangle}{\right\rangle}
\newcommand{\ol}{\overline}

\newtheorem{thm}{Theorem}[section]

\theoremstyle{plain}

\newtheorem{lem}[thm]{Lemma}
\newtheorem{cor}[thm]{Corollary}

\newtheorem*{theorem*}{Theorem}

\theoremstyle{definition}
\newtheorem{defn}[thm]{Definition}
\theoremstyle{remark}
\newtheorem{rem}[thm]{Remark}

\newtheorem{eg}[thm]{Example}

\title{Pirashvili--Richter-type theorems for the reflexive and dihedral homology theories}
\author{Daniel Graves}
\address{Lifelong Learning Centre, University of Leeds, Woodhouse, Leeds, LS2 9JT, UK}
\email{dan.graves92@gmail.com}
\date{}

\begin{document}



\maketitle

\begin{abstract}
Reflexive homology and dihedral homology are the homology theories associated to the reflexive and dihedral crossed simplicial groups respectively. The former has recently been shown to capture interesting information about $C_2$-equivariant homotopy theory and its structure is related to the study of ``real" objects in algebraic topology. The latter has long been of interest for its applications in $O(2)$-equivariant homotopy theory and connections to Hermitian algebraic $K$-theory. In this paper, we show that the reflexive and dihedral homology theories can be interpreted as functor homology over categories of non-commutative sets, after the fashion of Pirashvili and Richter's result for the Hochschild and cyclic homology theories.
\end{abstract}

\section*{Introduction}

\emph{Functor homology} plays an important role in modern homological algebra. The subject was pioneered by Connes (\cite{Connes}) who described the cyclic homology theory as derived functors over the cyclic category. This was extended to incorporate dihedral structure by Loday (\cite{Loday-dih}) and Krasauskas, Lapin and Solov'ev (\cite{KLS-dihed}). Collectively, this work led to the notion of a crossed simplicial group (see \cite{FL} and \cite{Kras-skew}). Crossed simplicial groups extended the notion of cyclic homology to detect other structure that is compatible with a unital, associative multiplication, such as an order-reversing involution.

A great many homology theories for algebras have interpretations in terms of functor homology: Hochschild homology (\cite[6.2.2]{Lod}); cyclic homology (\cite{Connes}, \cite[6.2.8]{Lod}); dihedral homology (\cite{Loday-dih}, \cite{KLS-dihed}); reflexive homology (\cite{DMG-reflexive}); symmetric homology (\cite{Ault}, \cite{Fie}); hyperoctahedral homology (\cite{DMG-HO}, \cite{Fie}); higher-order Hochschild homology (\cite{PirH}); gamma homology (\cite{PR00}); Andr\'e-Quillen homology (\cite{PirAQ}); $E_n$-homology (\cite{LR}); and Leibniz homology (\cite{HV}).   

The \emph{reflexive crossed simplicial group} and the \emph{dihedral crossed simplicial group} are two of the \emph{fundamental crossed simplicial groups}. The associated homology theories, \emph{reflexive homology} and \emph{dihedral homology}, are defined as functor homology over their respective indexing categories $\Delta R^{op}$ and $\Delta D^{op}$.

The dihedral homology theory has long been of interest for its applications to $O(2)$-equivariant homotopy theory and Hermitian algebraic $K$-theory (see \cite{Loday-dih}, \cite{KLS-dihed}, \cite{Dunn}, \cite{Lodder1}, \cite{Lodder3}, \cite{Cort1}, \cite{Cort2}, \cite{Lodder2}). By comparison, the study of the reflexive homology theory is a modern venture. Recent work of the author has studied the applications of the reflexive crossed simplicial group to $C_2$-equivariant homotopy theory and notes that the structure of the reflexive crossed simplicial group is used in the study of \emph{real} objects in algebraic topology (\cite{DMG-reflexive}).

From one point of view, reflexive homology and dihedral homology can be viewed as extensions of Hochschild homology and cyclic homology respectively, detecting the extra structure of an order-reversing involution.

The original functor homology interpretations of Hochschild homology and cyclic homology were given in terms of $\Delta^{op}$ and $\Delta C^{op}$, the categories indexing simplicial objects and cyclic objects respectively. Pirashvili and Richter (\cite{PR02}) gave an alternative functor homology interpretation of Hochschild homology and cyclic homology. They showed that Hochschild homology and cyclic homology can be expressed as functor homology over the \emph{category of non-commutative sets}, $\mathcal{F}(as)$. The category of non-commutative sets was introduced by Fe\u{\i}gin and Tsygan (\cite[A10]{FT}). It is the category that encapsulates the structure of a monoid in a symmetric monoidal category (\cite{Pir-PROP}) and it is isomorphic to $\Delta S$, the category associated to the symmetric crossed simplicial group (\cite{FL}, \cite{Ault}, \cite{Fie}).

The author introduced an extension of $\mathcal{F}(as)$, called the \emph{category of involutive non-commutative sets} and denoted $\mathcal{IF}(as)$. This category encapsulates the structure of an involutive monoid in a symmetric monoidal category (\cite{DMG-IFAS}) and is isomorphic to $\Delta H$, the category associated to the hyperoctahedral crossed simplicial group (\cite{FL}, \cite{DMG-IFAS}, \cite{DMG-HO}, \cite{Fie}).

In this paper we prove that, just as Hochschild and cyclic homology can be interpreted as functor homology over the category of non-commutative sets (equivalently, over $\Delta S$), reflexive homology and dihedral homology can be expressed as functor homology over the category of involutive non-commutative sets (equivalently, over $\Delta H$).

The paper is structured as follows.

In Sections \ref{functor-homology-sec} and \ref{crossed-simplicial-group-sec} we recall the necessary background on functor homology, crossed simplicial groups and their homology theories. In Section \ref{non-comm-sets-sec} we recall the category of involutive non-commutative sets, $\mathcal{IF}(as)$, and the subcategory of based involutive non-commutative sets, $\mathcal{I}\Gamma(as)$. In Section \ref{involutive-alg-sec}, we recall the special cases of functor homology for involutive algebras in terms of \emph{Loday functors} and \emph{bar constructions}.

In Sections \ref{Pir-Ric-sec} and \ref{Slom-Zimm-sec} we recall Pirashvili and Richter's identification of the Hochschild and cyclic homology theories as functor homology over the category of non-commutative sets. We also recall the category-theoretic framework introduced independently by S\l{}omi\'{n}ska and Zimmermann, which allows for an alternative proof of Pirashvili and Richter's result. This will be the framework we use to prove our results.

In Sections \ref{decomp-sec1} and \ref{decomp-sec2} we prove the technical results that we require for our theorems.

In Section \ref{dihedral-sec} we prove our first main theorem (Theorem \ref{dihedral-thm}) which, in the specific case of an involutive algebra (Corollary \ref{dihedral-cor}), can be stated as follows:

\begin{theorem*}
Let $A$ be an involutive, associative $k$-algebra. There exist isomorphisms of graded $k$-modules
\[HD_{\star}(A)=\mathrm{Tor}_{\star}^{\Delta D^{op}}\left(k^{\ast} , \mathcal{L}(A)\right) \cong \mathrm{Tor}_{\star}^{\ifas}\left(B, \mathsf{H}_A\right)\cong \mathrm{Tor}_{\star}^{\Delta H}\left(B, \mathsf{H}_A\right)\]
where $\mathcal{L}(A)$ is the dihedral Loday functor (also known as the dihedral bar construction); $\mathsf{H}_A$ is the hyperoctahedral bar construction; and the functor $B$ can be found in Definition \ref{functor-B-defn}.
\end{theorem*}

In Section \ref{reflexive-sec} we prove our second main theorem (Theorem \ref{reflexive-thm}) which, in the specific case of an involutive algebra (Corollary \ref{reflexive-cor}), can be stated as follows:

\begin{theorem*}
Let $A$ be an involutive, associative $k$-algebra and let $M$ be an involutive $A$-bimodule. There is an isomorphism of graded $k$-modules
\[HR_{\star}\left(A,M\right) =\mathrm{Tor}_{\star}^{\Delta R^{op}}\left(k^{\ast}, \mathcal{L}(A,M)\right) \cong  \mathrm{Tor}_{\star}^{\mathcal{I}\Gamma(as)}\left(B^{\prime} , \mathsf{R}(A,M)\right)\]
where $\mathcal{L}(A,M)$ is the reflexive Loday functor; $\mathsf{R}(A,M)$ is a restriction of the hyperoctahedral bar construction to the subcategory $\mathcal{I}\Gamma(as)$; and $B^{\prime}$ is the restriction of the functor $B$ to the subcategory $\mathcal{I}\Gamma(as)$.
\end{theorem*}

\subsection*{Acknowledgements}
I would like to thank Jake Saunders, Ieke Moerdijk and Sarah Whitehouse for interesting and helpful conversations whilst writing this paper. I would like to thank Andrew Fisher, Jack Davidson and James Brotherston for their thoughts and comments on an earlier version of this document.

\subsection*{Conventions}
Throughout the paper, we will let $k$ be a unital, commutative ring, $\mathbf{Mod}_k$ be the category of $k$-modules and $\otimes$ denote the tensor product of $k$-modules. All $k$-algebras in this paper will be assumed to be unital. For $n\geqslant 0$, we will let $[n]$ denote the set $\llb 0,1\dotsc , n\rrb$.

\section{Functor homology}
\label{functor-homology-sec}

We begin by recalling the categories of left and right modules over a small category. We also recall the tensor product of $\mathbf{C}$-modules and its left-derived functors (see \cite[Section 3]{MCM}, \cite[Section 1.6]{PR02} for instance).

\begin{defn}
For a small category $\mathbf{C}$ we define $\mathbf{CMod}=\mathrm{Fun}\left(\mathbf{C}, \mathbf{Mod}_k\right)$ and $\mathbf{ModC}=\mathrm{Fun}\left(\mathbf{C}^{op}, \mathbf{Mod}_k\right)$. 
\end{defn}

\begin{defn}
For a small category $\mathbf{C}$, we define the $k$-constant functor, $k^{\ast}$ in $\mathbf{CMod}$ to be the functor that sends every object of $\mathbf{C}$ to $k$ and every morphism in $\mathbf{C}$ to the identity map on $k$.
\end{defn}

\begin{defn}
\label{tensor-defn}
There is a tensor product
\[-\otimes_{\mathbf{C}} - \colon \mathbf{ModC} \times \mathbf{CMod} \rightarrow \mathbf{Mod}_k\]
defined as the coend
\[G\otimes_{\mathbf{C}} F = \int^{C\in \mathrm{Ob}(\mathbf{C})} G(C) \otimes F(C).\]

The left derived functors of the tensor product $-\otimes_{\mathbf{C}} -$ are denoted by $\mathrm{Tor}_{\star}^{\mathbf{C}}\left(-,-\right)$. 
\end{defn}

\section{Crossed simplicial groups}
\label{crossed-simplicial-group-sec}
We begin by recalling some notation for certain groups that we will use throughout the paper. We will recall the notion of crossed simplicial group together with examples. We will recall the functor homology definitions of reflexive homology, dihedral homology and hyperoctahedral homology.

\subsection{A gathering of groups}
\label{groups-subsec}

\begin{defn}
\label{groups-defn}
We define the groups that we will use throughout the paper. Let $n\geqslant 0$.
\begin{itemize}
\item The \emph{reflexive group} is defined by  
\[R_{n+1}=\llangle r_{n+1}\mid r_{n+1}^2=1\rrangle.\]
The generator $r_{n+1}$ acts on $[n]$ by $r_{n+1}(i)=n-i$.
\item Let 
\[C_{n+1}=\llangle t_{n+1} \mid t_{n+1}^{n+1}=1\rrangle\]
denote the \emph{cyclic group of order $n+1$}. We note that $C_{n+1}$ acts on the set $[n]$ as follows: $t_{n+1}(i)=i+1$ for $0\leqslant i \leqslant n-1$ and $t_{n+1}(n)=0$.
\item Let  
\[D_{n+1}=\llangle r_{n+1},\,t_{n+1} \mid r_{n+1}^2=t_{n+1}^{n+1}=1,\, r_{n+1}t_{n+1}r_{n+1}=t_{n+1}^{-1}\rrangle.\]
denote the \emph{dihedral group}. We note that $D_{n+1}$ acts on the set $[n]$. The generator $t_{n+1}$ acts as for the cyclic group $C_{n+1}$. The generator $r_{n+1}$ acts as for $R_{n+1}$.
\item We denote by $\Sigma_{n+1}$ the \emph{symmetric group} on the set $[n]$. 
\item Let $\Sigma_{n+1}^{+}$ be the subgroup of permutations which fix $0$.
\item The \emph{hyperoctahedral group} $H_{n+1}$ is the semi-direct product $C_2^{n+1} \rtimes \Sigma_{n+1}$ where $\Sigma_{n+1}$ acts on $C_2^{n+1}$ by permuting the factors. 
\item Let $H_{n+1}^{+}$ denote the subgroup of $H_{n+1}$ consisting of elements $(z_0,\, z_1,\dotsc,z_n;\sigma)$ such that $z_0=1\in C_2$ and $\sigma \in \Sigma_{n+1}^{+}$.
\end{itemize}
\end{defn}

The following lemma gives us a specific instance of the dihedral group inside the hyperoctahedral group. This will be of use for our technical results in Section \ref{decomp-sec1}.

\begin{lem}
\label{dihedral-subgroup-lem}
The elements $R=\left(t_2,\dotsc, t_2;r_{n+1}\right)$ and $T=\left(1,\dotsc , 1; t_{n+1}\right)$ generate a subgroup of $H_{n+1}$ isomorphic to $D_{n+1}$.
\end{lem}
\begin{proof}
We see that $R^2=\left(t_2,\dotsc, t_2;r_{n+1}\right)^2=1=\left(1,\dotsc , 1; t_{n+1}\right)^{n+1} =T^{n+1}$. We also see that 
\[RTR=\left(t_2^2,\dotsc ,t_2^2;r_{n+1}t_{n+1}r_{n+1}\right)=\left(1,\dotsc, 1;t_{n+1}^{-1}\right)=T^{-1}\]
which completes the proof.
\end{proof}

\subsection{Crossed simplicial groups}

The theory of crossed simplicial groups was developed independently by Fiedorowicz and Loday (\cite{FL}) and Krasauskas (\cite{Kras-skew}), motivated by Connes' construction of cyclic homology (\cite{Connes}).

\begin{defn}
The category $\Delta$ has as objects the sets $[n]$ for $n\geqslant 0$, with order-preserving maps as morphisms.
\end{defn}

We recall, from \cite[1.1]{FL}, that a family of groups $\llb G_n\rrb$ for $n\geqslant 0$ is called a \emph{crossed simplicial group} if there is a category $\Delta G$ such that the objects are the sets $[n]$; $\Delta$ is a subcategory of $\Delta G$; $\mathrm{Aut}_{\Delta G}([n])=G_n$; and any morphism $[n]\rightarrow [m]$ in $\Delta G$ can be uniquely written as a pair $\left(\phi , g\right)$ with $g\in G_n$ and $\phi\colon [n]\rightarrow [m]$ in $\Delta$.

\begin{eg}
\label{CSG-egs}
We recall the \emph{fundamental crossed simplicial groups}. 
\begin{itemize}
\item If we take the trivial group for each $n$, we recover the category $\Delta$.
\item If we take the family of reflexive groups, $\llb R_{n+1}\rrb$ (see \cite[Example 2]{FL} and \cite[1.1]{DMG-reflexive}), we obtain the \emph{reflexive category} $\Delta R$, (see \cite[1.2]{DMG-reflexive}).
\item If we take the family of cyclic groups, $\llb C_{n+1} \rrb$ (see \cite[Example 4]{FL}), we obtain the \emph{cyclic category} $\Delta C$, which is the same as Connes' category $\Lambda$ (\cite{Connes}).
\item If we take the family of dihedral groups, $\llb D_{n+1}\rrb$ (see \cite[Example 4]{FL}), we obtain the \emph{dihedral category} $\Delta D$, (see \cite[Section 3]{Loday-dih}, and also \cite[1.1]{KLS-dihed}, where this category is denoted by $\Xi$).
\item If we take the family of symmetric groups, $\llb \Sigma_{n+1}\rrb$ (see \cite[Example 6]{FL}), we obtain the \emph{symmetric category} $\Delta S$ (see \cite[1.1]{Ault}).
\item If we take the family of hyperoctahedral groups, $\llb H_{n+1}\rrb$ (see \cite[Example 6]{FL}), we obtain the \emph{hyperoctahedral category} $\Delta H$ (see \cite[1.2]{DMG-HO}).
\end{itemize}
\end{eg}

\subsection{Functor homology theories}

Given a crossed simplicial group, we have an associated functor homology theory. In this subsection we recall the definition of reflexive homology from \cite{DMG-reflexive}; the definition of dihedral homology from \cite{Loday-dih} and \cite{KLS-dihed}; and the definition of hyperoctahedral homology from \cite{Fie} and \cite{DMG-HO}. 

\begin{defn}
Let $F\colon \Delta R^{op}\rightarrow \mathbf{Mod}_k$. The \emph{reflexive homology of $F$} is defined to be
\[HR_{\star}(F)=\mathrm{Tor}_{\star}^{\Delta R^{op}}\left(k^{\ast} , F\right).\]
\end{defn}

\begin{defn}
Let $F\colon \Delta D^{op}\rightarrow \mathbf{Mod}_k$. The \emph{dihedral homology of $F$} is defined to be
\[HD_{\star}(F)=\mathrm{Tor}_{\star}^{\Delta D^{op}}\left(k^{\ast} , F\right).\]
\end{defn}

It was shown by Fiedorowicz and Loday \cite[6.16]{FL} that the functor homology theory for $\Delta H^{op}$ is isomorphic to Hochschild homology. However, the functor homology theory associated to $\Delta H$ is very different. As shown in previous work of the author (\cite{DMG-HO}, \cite{DMG-e-inf}), inspired by Fiedorowicz's preprint (\cite{Fie}), the functor homology theory for $\Delta H$ has applications in $C_2$-equivariant stable homotopy theory. Whilst hyperoctahedral homology does not appear directly in this paper, functor homology over the category $\Delta H$ does, so we include the definition for reference.

\begin{defn}
Let $F\colon \Delta H\rightarrow \mathbf{Mod}_k$. The \emph{hyperoctahedral homology of $F$} is defined to be 
\[HO_{\star}(F)=\mathrm{Tor}_{\star}^{\Delta H}(k^\ast, F).\]
\end{defn}

\subsection{A gathering of groupoids}

We will want to consider crossed simplicial groups not only as a family of groups but also as a groupoid. We do this in the natural way: the objects are the sets $[n]$ for $n\geqslant 0$ and the automorphisms are the group elements in degree $n$. Recall the groups defined in Subsection \ref{groups-subsec}.

\begin{defn}
\label{groupoids-defn}
We will make use of the following groupoids.
\begin{itemize}
\item Let $\mathbf{H}$ denote the \emph{groupoid of hyperoctahedral groups}. The objects are the sets $[n]$ and $\mathrm{Hom}_{\mathbf{H}}\left([n],[n]\right)=H_{n+1}$.
\item Let $\mathbf{H}^{+}$ denote the sub-groupoid of $\mathbf{H}$ such that $\mathrm{Hom}_{\mathbf{H}^{+}}\left([n],[n]\right)=H_{n+1}^{+}$.
\item Let $\mathbf{D}$ denote the sub-groupoid of $\mathbf{H}$ such that $\mathrm{Hom}_{\mathbf{D}}\left([n],[n]\right)$ is the dihedral subgroup of $H_{n+1}$ generated by the elements $R$ and $T$ of Lemma \ref{dihedral-subgroup-lem}.
\item Let $\mathbf{R}$ denote the sub-groupoid of $\mathbf{H}$ such that $\mathrm{Hom}_{\mathbf{R}}\left([n],[n]\right)$ is the reflexive subgroup of $H_{n+1}$ generated by the element $R$ of Lemma \ref{dihedral-subgroup-lem}.
\end{itemize}
\end{defn}

\section{Involutive non-commutative sets}
\label{non-comm-sets-sec}

We recall the categories of non-commutative sets and involutive non-commutative sets together with the subcategories of based non-commutative sets and based, involutive, non-commutative sets.

\subsection{The category of non-commutative sets}

\begin{defn}
The \emph{category of non-commutative sets}, $\fas$, has as objects the sets $[n]$ for $n\geqslant 0$. An element $f\in\mathrm{Hom}_{\fas}([n],[m])$ is a map of finite sets such that $f^{-1}(i)$ is a totally ordered set for each $i\in [m]$. Given composable morphisms $f\in\mathrm{Hom}_{\fas}([n],[m])$ and $g\in\mathrm{Hom}_{\fas}([m],[k])$, the composite is given by the underlying map of finite sets $g\circ f$, together with the preimage data given by the ordered disjoint union:
\[\left(g\circ f\right)^{-1}(i) = \coprod_{j \in g^{-1}(i)} f^{-1}(j).\]
\end{defn}

\begin{defn}
The \emph{category of based non-commutative sets}, $\Gamma(as)$, is the subcategory of $\fas$ with the same set of objects and whose morphisms satisfy $f(0)=0$.
\end{defn}

\begin{rem}
We note that the version of the category $\Gamma(as)$ that we use here has a total ordering on the preimage of the basepoint. There are versions of this category that forgo a total-ordering on the preimage of the basepoint. See the paper of Hartl, Pirashvili and Vespa \cite[Example 2.11]{HPV} for example.
\end{rem}

\subsection{The category of involutive non-commutative sets}

\begin{defn}
A \emph{finite, totally ordered $C_2$-set} consists of a finite set whose elements have a total ordering together with a superscript label from the group $C_2$ for each element.
\end{defn}

\begin{defn}
\label{action-defn}
Let $t$ denote the generator of $C_2$. We define an action of $C_2$ on finite, totally-ordered $C_2$-sets by
\[ t \ast \llb j_1^{\alpha_{j_1}}<\cdots <j_r^{\alpha_{j_r}}\rrb=\llb j_r^{t\alpha_{j_r}}<\cdots <j_1^{t\alpha_{j_1}}\rrb.\]
Explicitly, we invert the ordering and multiply each label by $t\in C_2$.
\end{defn}

\begin{defn}
The \emph{category of involutive, non-commutative sets}, $\ifas$, has as objects the sets $[n]$ for $n\geqslant 0$. An element $f\in\mathrm{Hom}_{\ifas}([n],[m])$ is a map of finite sets such that $f^{-1}(i)$ is a totally ordered $C_2$-set for each $i\in [m]$. Given composable morphisms $f\in\mathrm{Hom}_{\ifas}([n],[m])$ and $g\in\mathrm{Hom}_{\ifas}([m],[k])$, the composite is given by the underlying map of finite sets $g\circ f$, together with the preimage data given by the ordered disjoint union:
\[\left(g\circ f\right)^{-1}(i) = \coprod_{j^\alpha \in g^{-1}(i)} \alpha \ast f^{-1}(j)\]
with the action from Definition \ref{action-defn}.
\end{defn}

\begin{rem}
We recall from \cite[1.7]{DMG-HO} (see also \cite[3.11]{DMG-IFAS}) that there is an isomorphism of categories $\Delta H \cong \mathcal{IF}(as)$.
\end{rem}

\begin{defn}
The \emph{category of based, involutive, non-commutative sets}, $\mathcal{I}\Gamma(as)$, is the subcategory of $\ifas$ with the same objects and whose morphisms satisfy $f(0)=0$.
\end{defn}

\section{Involutive algebras, Loday functors and bar constructions}
\label{involutive-alg-sec}

We recall the definition of an involutive $k$-algebra. We also recall the definition of an involutive bimodule (see \cite[5.2.1]{Lod} for instance). We recall the Loday functors for the reflexive homology and dihedral homology of an involutive $k$-algebra. We also recall the hyperoctahedral bar construction, used to define the hyperoctahedral homology of an involutive $k$-algebra.  

\subsection{Involutive algebras}

\begin{defn}
An \emph{involutive $k$-algebra} consists of an associative $k$-algebra $A$ together with an order-reversing $k$-algebra automorphism $i\colon A\rightarrow A$ which squares to the identity. This is commonly written as a map $a \mapsto \overline{a}$ satisfying the following conditions for $a,\,b \in A$ and $\lambda \in k$:
\begin{multicols}{2}
\begin{itemize}
\item $\overline{a+b}=\overline{a} +\overline{b}$;
\item $\overline{ab}=\overline{b}\, \overline{a}$;
\item $\overline{\overline{a}}=a$;
\item $\overline{\lambda}=\lambda$.
\end{itemize}
\end{multicols}
\end{defn}

\begin{defn}
Let $A$ be an involutive $k$-algebra. An \emph{involutive $A$-bimodule} is an $A$-bimodule $M$ equipped with a map $m\mapsto \ol{m}$ such that $\ol{a_1ma_2}= \ol{a_2}\, \ol{m}\, \ol{a_1}$ for $a_1$, $a_2\in A$.
\end{defn}

\subsection{Loday functors and homology theories for involutive algebras}

\begin{defn}
\label{Loday-func-Hochschild-defn}
For an associative $k$-algebra $A$ and an $A$-bimodule $M$, the \emph{Loday functor} is the  simplicial $k$-module 
\[\mathcal{L}(A,M)\colon \Delta^{op}\rightarrow \mathbf{Mod}_k\]
given on objects by $[n]\mapsto M\otimes A^{\otimes n}$ and determined on morphisms by 
\[\partial_i\left(m\otimes a_1\otimes \cdots \otimes a_n\right) =
\begin{cases}
\left( ma_1\otimes a_2\otimes \cdots \otimes a_n\right) & i=0\\
\left(m\otimes a_1\otimes \cdots \otimes a_ia_{i+1}\otimes \cdots \otimes a_n\right) & 1\leqslant i \leqslant n-1\\
\left(a_nm\otimes a_1\otimes \cdots \otimes a_{n-1}\right) & i=n
\end{cases}
\]
and
\[s_j\left(m\otimes a_1\otimes \cdots \otimes a_n\right) = 
\begin{cases}
\left(m\otimes 1_A \otimes a_1\otimes \cdots \otimes a_n\right) &j=0\\
\left(m\otimes a_1\otimes \cdots\otimes a_j \otimes 1_A \otimes a_{j+1} \otimes \cdots \otimes a_n\right) & j\geqslant 1.
\end{cases}
\]
\end{defn}

Following \cite[1.8, 1.9]{DMG-reflexive} we can extend the Loday functor to a functor on $\Delta R^{op}$ and therefore define the reflexive homology of an involutive algebra with coefficients in an involutive bimodule. (We note that this functor was denoted by $\mathcal{L}^{+}(A,M)$ and the homology was denoted by $HR_{\star}^{+}(A,M)$ in that paper.)

\begin{defn}
\label{loday-defn}
Let $A$ be an involutive $k$-algebra and let $M$ be an involutive $A$-bimodule. We extend the Loday functor $\mathcal{L}(A,M)$ to a functor 
\[\mathcal{L}\left(A,M\right)\colon \Delta R^{op}\rightarrow \mathbf{Mod}_k\] by defining
\[r_{n+1}\left(m\otimes a_1\otimes \cdots \otimes a_n\right) = \left(\ol{m}\otimes \ol{a_n}  \otimes \cdots \otimes \ol{a_1}\right).\]
We denote the \emph{reflexive homology of $A$ with coefficients in $M$} by
\[HR_{\star}(A,M)=\mathrm{Tor}_{\star}^{\Delta R^{op}}\left(k^{\ast},\mathcal{L}(A,M)\right).\]
\end{defn}

We can extend the Loday functor for $\Delta R^{op}$ to a functor on $\Delta D^{op}$. This first appeared in \cite[Section 2]{Loday-dih} and \cite[Section 1]{KLS-dihed} (see also \cite[1.4]{Lodder1}, \cite[Section 2]{Dunn} for example) and is also sometimes called the \emph{dihedral bar construction}. We note that in this case we have introduced cycles so we must have $M=A$ and we write $\mathcal{L}(A)$ rather than $\mathcal{L}(A,A)$.

\begin{defn}
\label{dihedral-loday}
Let $A$ be an involutive $k$-algebra. We extend the Loday functor $\mathcal{L}(A)$ from $\Delta R^{op}$ to a functor 
\[\mathcal{L}\left(A\right)\colon \Delta D^{op}\rightarrow \mathbf{Mod}_k\] by defining
\[t_{n+1}\left(a_0\otimes a_1\otimes \cdots \otimes a_n\right) = \left(a_n\otimes a_0 \otimes \cdots \otimes a_{n-1}\right).\]
We denote the \emph{dihedral homology of $A$} by
\[HD_{\star}(A)=\mathrm{Tor}_{\star}^{\Delta D^{op}}\left(k^{\ast},\mathcal{L}(A)\right).\]
\end{defn}

\subsection{The hyperoctahedral bar construction}

\begin{defn}
\label{hyperoctahedral-bar-defn}
Let $A$ be an involutive $k$-algebra. The \emph{hyperoctahedral bar construction}
\[\mathsf{H}_A\colon \Delta H\rightarrow \mathbf{Mod}_k\]
is given on objects by $[n]\mapsto A^{\otimes n+1}$. Let $(\phi ,g )\in \mathrm{Hom}_{\Delta H}\left([n],[m]\right)$ (so $\phi \in \mathrm{Hom}_{\Delta}\left([n] , [m]\right)$ and $g=\left( z_0,\dotsc , z_n; \sigma\right) \in H_{n+1}$). We define $\mathsf{H}_A\left(\phi ,g\right)$ to be determined $k$-linearly by
\[\mathsf{H}_A\left(\phi ,g\right)(a_0\otimes \cdots \otimes a_n)=b_0\otimes \cdots \otimes b_m,\]
where 
\[b_i=\prod_{j\in (\phi\circ \sigma)^{-1}(i)} a_j^{z_j}\]
where the product is ordered according to the map $\phi$ and
\[a_j^{z_j}=
\begin{cases}
a_j & z_j=1\\
\ol{a_j} & z_j=t_2.
\end{cases}\]
Note that an empty product is defined to be the multiplicative unit $1_A$.
\end{defn}

\subsection{Bar constructions on categories of non-commutative sets}

\begin{defn}
Let $A$ be an involutive, associative $k$-algebra.  By abuse of notation, we will write $\mathsf{H}_A\colon \mathcal{IF}(as)\rightarrow \mathbf{Mod}_k$ for the functor obtained by pre-composing the hyperoctahedral bar construction with the isomorphism of \cite[1.7]{DMG-HO}. In this setting, the functor $\mathsf{H}_A$ is defined as follows. On objects we define $\mathsf{H}_A([n])=A^{\otimes n+1}$.

Let $a_0\otimes \cdots \otimes a_n \in A^{\otimes n+1}$. Let $f\colon [n]\rightarrow [m]$ in $\mathcal{IF}(as)$. We define $\mathsf{H}_A(f)$ to be determined $k$-linearly by
\[\mathsf{H}_A(f)\left(a_0\otimes \cdots \otimes a_n\right) = b_0\otimes \cdots \otimes b_m,\]
where 
\[b_i=\prod_{j\in f^{-1}(i)} a_j^{z_j}\]
with the product ordered according to the total ordering on $f^{-1}(i)$ and $a_j^{z_j}$ given as in Definition \ref{hyperoctahedral-bar-defn}. 
\end{defn}

\begin{defn}
Let $A$ be an involutive, associative $k$-algebra and let $M$ be an involutive $A$-bimodule. We define a functor $\mathsf{R}(A,M)\colon \mathcal{I}\Gamma(as) \rightarrow \mathbf{Mod}_k$ as follows. On objects we define $\mathsf{R}(A,M)([n])=M\otimes A^{\otimes n}$.

Let $a_0\otimes \cdots \otimes a_n \in M\otimes A^{\otimes n}$. In particular, note that $a_0\in M$. Let $f\colon [n]\rightarrow [m]$ in $\mathcal{I}\Gamma(as)$. We define $\mathsf{R}(A,M)(f)$ to be determined $k$-linearly by
\[\mathsf{R}(A,M)(f)\left(a_0\otimes \cdots \otimes a_n\right) = b_0\otimes \cdots \otimes b_m,\]
where 
\[b_i=\prod_{f^{-1}(i)} a_j^{z_i}\]
with the product ordered according to the total ordering on $f^{-1}(i)$ and $a_j^{z_j}$ given as in Definition \ref{hyperoctahedral-bar-defn}.
\end{defn}

\begin{rem}
We observe that if we take $M=A$, then the functor $\mathsf{R}(A,A)$ is the restriction of the hyperoctahedral bar construction $\mathsf{H}_A$ to the subcategory $\mathcal{I}\Gamma(as)\subset \mathcal{IF}(as)$. 
\end{rem}

\section{The results of Pirashvili and Richter}
\label{Pir-Ric-sec}

We give a slight reformulation of Pirashvili and Richter's theorem.

\begin{thm}[{{\cite[1.3]{PR02}}}]
For functors $F\colon \Gamma(as)\rightarrow \mathbf{Mod}_k$ and $T\colon \mathcal{F}(as)\rightarrow \mathbf{Mod}_k$, there are isomorphisms of graded $k$-modules
\[\mathrm{Tor}_{\star}^{\Delta C^{op}}\left(k^{\ast} , T\circ i_1\right) \cong \mathrm{Tor}_{\star}^{\mathcal{F}(as)}\left(b , T\right)\quad \text{and} \quad \mathrm{Tor}_{\star}^{\Delta^{op}}\left(k^{\ast} , F\circ i_2\right) \cong \mathrm{Tor}_{\star}^{\Gamma(as)}\left(\overline{b} , F\right),\]
where $i_1\colon \Delta C^{op}\rightarrow \mathcal{F}(as)$ and $i_2\colon \Delta^{op}\rightarrow \Gamma(as)$ are inclusions of subcategories; the functor $b$ was defined in \cite[1.7]{PR02} and $\overline{b}$ is the restriction of $b$ to the subcategory $\Gamma(as)$.
\end{thm}

Pirashvili and Richter proved their theorem by using the axiomatic characterization of $\mathrm{Tor}$ functors, which required detailed analysis of certain cyclic sets. However, it was shown by S\l{}omi\'{n}ska and Zimmermann that there theorems can be proved using category-theoretic methods.  

\section{The framework of S\l{}omi\'{n}ska and Zimmermann}
\label{Slom-Zimm-sec}

S\l{}omi\'{n}ska (\cite{Slom}) and Zimmermann (\cite{Zimm}) independently developed a categorical framework that generalizes the theorems of Pirashvili and Richter (\cite{PR02}).

Let $\mathbf{C}$ be a small category and let $\mathbf{A}$ and $\mathbf{B}$ be subcategories of $\mathbf{C}$, such that all three categories have the same set of objects. Following S\l{}omi\'{n}ska, we will write $\mathbf{C}=\mathbf{A}\circ \mathbf{B}$ if every morphism $f$ in $\mathbf{C}$ can be uniquely written as a composite, $f_{\mathbf{A}} \circ f_{\mathbf{B}}$, of a morphism in $\mathbf{B}$ followed by a morphism in $\mathbf{A}$. We note that we would write $\mathbf{C}=\mathbf{A} \Join \mathbf{B}$ in the notation of Zimmermann.

\begin{eg}
The following important examples follow from the very definition of a crossed simplicial group. Recall the crossed simplicial group categories from Example \ref{CSG-egs} and the groupoids of Definition \ref{groupoids-defn}. We have decompositions $\Delta R = \Delta \circ \mathbf{R}$; $\Delta C = \Delta \circ \mathbf{C}$; $\Delta D=\Delta \circ \mathbf{D}$; $\Delta S= \Delta \circ \mathbf{S}$; and $\Delta H = \Delta \circ \mathbf{H}$.
\end{eg}

\section{Decomposing the category of involutive non-commutative sets}
\label{decomp-sec1}

In this section we prove the technical results required to prove Theorem \ref{dihedral-thm}. This involves decomposing the category $\mathcal{IF}(as)$ into the dihedral category $\Delta D^{op}$ and the groupoid $\mathbf{H}^{+}$. 

\begin{lem}
\label{ifas-decomp-lem}
There is a decomposition $\ifas = \Delta \circ \mathbf{H}$.
\end{lem}
\begin{proof}
As shown in \cite[1.7]{DMG-HO} (see also \cite[3.11]{DMG-IFAS}), there is an isomorphism of categories $\ifas \cong \Delta H$ from which the result follows.
\end{proof}

\begin{lem}
\label{hyp-groupoid-decomp}
There is a decomposition $\mathbf{H}=\mathbf{D} \circ \mathbf{H}^{+}$.
\end{lem}
\begin{proof}
Let $\left(z_0,\dotsc , z_n;\sigma\right)$ be an element of $H_{n+1}$. For ease of notation, we will write $\tau$ for $t_{n+1}$; $t$ for $t_2$; and $r$ for $r_{n+1}$. Suppose $\sigma(0)=k$. We can therefore write $\sigma=\tau^k\tau^{-k}\sigma$, where $\tau^{-k}\sigma$ is a permutation in $\Sigma_{n+1}^{+}$. We now have two cases. On the one hand, suppose that $z_k=1\in C_2$. In this case, we can write
\[ \left(z_0,\dotsc , z_n;\sigma\right) = \left(1,\dotsc , 1;\tau^k\right) \circ \left(1,z_{k+1},\dotsc , z_{k-1}; \tau^{-k}\sigma\right)\]
and we see that $\left(1,\dotsc , 1;\tau^k\right)\in D_{n+1}$ and $\left(1,z_{k+1},\dotsc , z_{k-1}; \tau^{-k}\sigma\right)\in H_{n+1}^{+}$ as required.

On the other hand, suppose $z_k=t_2=t \in C_2$. In this case we can write
\[\left(z_0,\dotsc , z_n;\sigma\right) = \left(t,\dotsc , t;r\right)\circ \left(t z_n,\dotsc, t z_0;r\sigma\right).\]
Now, $t z_k=1$ and this occurs in position $n-k$, so we can apply a similar argument to the previous case to write
\begin{align*}
\left(z_0,\dotsc , z_n;\sigma\right) &= \left(t,\dotsc , t;r\right)\circ \left(t z_n,\dotsc, t z_0;r\sigma\right)\\
&= \left(t,\dotsc , t;r\right)\circ \left(1,\dotsc , 1;\tau^{n-k}\right) \circ \left(1, t z_{k-1},\dotsc , t z_{k+1} ;\tau^{-(n-k)}r\sigma\right).
\end{align*}
We observe that 
\[\tau^{-(n-k)}r\sigma(0)=\tau^{-(n-k)}r(k)=\tau^{-(n-k)}(n-k)=0\]
so $\left(1, t z_{k-1},\dotsc , t z_{k+1} ;\tau^{-(n-k)}r\sigma\right)\in H_{n+1}^{+}$ and $\left(t,\dotsc , t;r\right)\circ \left(1,\dotsc , 1;\tau^{n-k}\right)\in D_{n+1}$ as required.
\end{proof}

\section{Decomposing the category of based, involutive non-commutative sets}
\label{decomp-sec2}
In this section we prove the technical results required to prove Theorem \ref{reflexive-thm}. This involves decomposing the category $\mathcal{I}\Gamma(as)$ into the reflexive category $\Delta R^{op}$ and the groupoid $\mathbf{H}^{+}$. 

\begin{lem}
\label{Delta-R-subcat-lem}
The category $\Delta R^{op}$ is a subcategory of $\mathcal{I}\Gamma(as)$.
\end{lem}
\begin{proof}
We begin by noting that $\Delta^{op}$ is a subcategory of $\mathcal{I}\Gamma(as)$. In \cite[1.4]{PR02}, it is shown that $\Delta^{op}$ is a subcategory of $\Gamma(as)$ by explaining how to include face and degeneracy maps. We can include the category $\Delta^{op}$ into $\mathcal{I}\Gamma(as)$ as a subcategory by adding identity labels to all preimages of singletons. 

We include the reflections $r_{n+1}$ for $n\geqslant 0$ from $\Delta R^{op}$ into $\mathcal{I}\Gamma(as)$ as follows. For the reflection
\[r_{n+1}\colon [n]\rightarrow [n]\]
we write
\[r_{n+1}^{-1}(i)=\begin{cases}
\llb 0^t\rrb & i=0\\
\llb (n-i+1)^t\rrb & i>0
\end{cases}
\]
in $\mathcal{I}\Gamma(as)$. A straightforward check shows that the generators of $\Delta^{op}$ and the maps $r_{n+1}$ generate the category $\Delta R^{op}$ as a subcategory of $\mathcal{I}\Gamma(as)$.
\end{proof}

\begin{lem}
\label{I-gamma-as-decomp-lem}
There is a decomposition $\mathcal{I}\Gamma(as) = \Delta R^{op} \circ \mathbf{H}^{+}$.
\end{lem}
\begin{proof}
By Lemmata \ref{ifas-decomp-lem} and \ref{hyp-groupoid-decomp} we have a decomposition $\ifas = \left(\Delta D\right) \circ \mathbf{H}^{+}$.
The groupoid $\mathbf{H}^{+}$ is a subcategory of $\mathcal{I}\Gamma(as)$, since the elements of this groupoid preserve $0$. We also know that $\mathcal{I}\Gamma(as)$ is a subcategory of $\ifas$. Combining these facts, together with the duality isomorphism $\Delta D \cong \Delta D^{op}$ of \cite[1.4]{Dunn}, we obtain
\begin{align*}
\mathcal{I}\Gamma(as) &= \ifas \cap \mathcal{I}\Gamma(as)\\
&=\left(\Delta D \circ \mathbf{H}^{+}\right) \cap \mathcal{I}\Gamma(as)\\
&= \left(\Delta D \cap \mathcal{I}\Gamma(as)\right)\circ \mathbf{H}^{+}\\
&\cong \left(\Delta D^{op} \cap \mathcal{I}\Gamma(as)\right)\circ \mathbf{H}^{+}.
\end{align*}
By Lemma \ref{Delta-R-subcat-lem}, $\Delta R^{op}$ is a subcategory of $\mathcal{I}\Gamma(as)$ and so $\Delta R^{op}$ is a subcategory of the intersection $\Delta D^{op} \cap \mathcal{I}\Gamma(as)$. Now, consider a morphism in $\Delta D^{op}$ which is not in $\Delta R^{op}$. Since $\Delta R^{op}$ contains all the order-preserving maps and all the reflections, such a morphism must contain a non-identity cycle. However, such a map does not preserve $0$ and so cannot be an morphism in $\mathcal{I}\Gamma(as)$. Therefore, we have $\Delta R^{op}=\left(\Delta D^{op} \cap \mathcal{I}\Gamma(as)\right)$ and so $\mathcal{I}\Gamma(as) = \Delta R^{op} \circ \mathbf{H}^{+}$ as required.
\end{proof}

\section{Dihedral homology}
\label{dihedral-sec}
We use the decompositions of Section \ref{decomp-sec1} to prove our first main theorem. 

\begin{defn}
\label{functor-B-defn}
Let $B\colon \ifas^{op}\rightarrow \mathbf{Mod}_k$ be the functor defined on objects by
\[B([n])=k\left[\coprod_{m\geqslant 0} \mathrm{Hom}_{\mathbf{H}^{+}}\left([m],[n]\right)\right]\]
and determined on morphisms by pre-composition. By abuse of notation we will also let $B$ denote the functor $\Delta H^{op}\rightarrow \mathbf{Mod}_k$ obtained by pre-composing with the opposite of the isomorphism of \cite[1.7]{DMG-HO}.
\end{defn}

\begin{thm}
\label{dihedral-thm}
Let $F\colon \ifas \rightarrow \mathbf{Mod}_k$. There exist isomorphisms of graded $k$-modules
\[HD_{\star}(F\circ i\circ D)=\mathrm{Tor}_{\star}^{\Delta D^{op}}\left(k^{\ast}, F\circ i\circ D\right) \cong\mathrm{Tor}_{\star}^{\ifas}\left(B,F\right)\cong \mathrm{Tor}_{\star}^{\Delta H}\left(B,F\right),\]
where $D\colon\Delta D^{op}\rightarrow \Delta D$ is the dihedral duality isomorphism of \cite[1.4]{Dunn} and $i\colon \Delta D \rightarrow \ifas$ is the inclusion of the subcategory.
\end{thm}
\begin{proof}
By Lemmata \ref{ifas-decomp-lem} and \ref{hyp-groupoid-decomp}, we have decompositions
\[\ifas=\Delta \circ \mathbf{H}= \Delta \circ \mathbf{D} \circ \mathbf{H}^{+}=\left(\Delta \circ \mathbf{D}\right)\circ \mathbf{H}^{+},\]
where $\Delta \circ \mathbf{D}$ is precisely the dihedral category $\Delta D$.

By \cite[1.1]{Slom} or \cite[2.7]{Zimm}, we have an isomorphism of graded $k$-modules 
\[\mathrm{Tor}_{\star}^{\ifas}\left(B,F\right) \cong \mathrm{Tor}_{\star}^{\Delta D}\left(k^{\ast}, F\circ i\right),\]
where $k^{\ast}$ is the $k$-constant right $\Delta D$-module. The first isomorphism follows from the fact that $k^{\ast}\circ D$ is the $k$-constant right $\Delta D^{op}$-module, which we also call $k^{\ast}$. The second isomorphism follows from the isomorphism of \cite[1.7]{DMG-HO}.
\end{proof}

\begin{cor}
\label{dihedral-cor}
Let $A$ be an involutive $k$-algebra. There exist isomorphisms of graded $k$-modules
\[HD_{\star}(A)=\mathrm{Tor}_{\star}^{\Delta D^{op}}\left(k^{\ast} , \mathcal{L}(A)\right) \cong \mathrm{Tor}_{\star}^{\ifas}\left(B, \mathsf{H}_A\right)\cong \mathrm{Tor}_{\star}^{\Delta H}\left(B, \mathsf{H}_A\right).\]
\end{cor}
\begin{proof}
By \cite[Lemma 2.2]{Fie}, the composite $\mathsf{H}_A\circ i \circ D\colon \Delta D^{op}\rightarrow \mathbf{Mod}_k$ is the Loday functor $\mathcal{L}(A)$ of Definition \ref{dihedral-loday} and the result follows from Theorem \ref{dihedral-thm}.
\end{proof}

\begin{rem}
Pirashvili and Richter's result was the first to draw a link between the category of non-commutative sets and cyclic homology. This connection has been shown to go further, notably in papers of Angeltveit (\cite{Angeltveit1}, \cite{Angeltveit2}), where the category of non-commutative sets plays an important role in forming a cyclic bar construction for $A_{\infty}$ $H$-spaces. In future work, the author hopes to investigate further connections between the category of involutive, non-commutative sets and dihedral homology.
\end{rem}

\section{Reflexive homology}
\label{reflexive-sec}
We use the decomposition of Section \ref{decomp-sec2} to prove our second main theorem.

\begin{thm}
\label{reflexive-thm}
Let $F\colon \mathcal{I}\Gamma(as)\rightarrow \mathbf{Mod}_k$. There is an isomorphism of graded $k$-modules
\[HR_{\star}(F\circ i) = \mathrm{Tor}_{\star}^{\Delta R^{op}}\left(k^{\ast}, F\circ i\right) \cong\mathrm{Tor}_{\star}^{\mathcal{I}\Gamma(as)}\left(B^{\prime},F\right)\]
where $B^{\prime}$ is the restriction of the functor $B$ from Definition \ref{functor-B-defn} to the subcategory $\mathcal{I}\Gamma(as)$ and $i\colon \Delta R^{op}\rightarrow \mathcal{I}\Gamma(as)$ is the inclusion of the subcategory.
\end{thm}
\begin{proof}
By Lemma \ref{I-gamma-as-decomp-lem}, we have a decomposition $\mathcal{I}\Gamma(as) = \left(\Delta R^{op}\right)\circ \mathbf{H}^{+}$. The theorem now follows from \cite[1.1]{Slom} or \cite[2.7]{Zimm}.
\end{proof}

\begin{cor}
\label{reflexive-cor}
Let $A$ be an involutive $k$-algebra and let $M$ be an involutive $A$-bimodule. There is an isomorphism of graded $k$-modules
\[HR_{\star}\left(A,M\right) =\mathrm{Tor}_{\star}^{\Delta R^{op}}\left(k^{\ast}, \mathcal{L}(A,M)\right) \cong  \mathrm{Tor}_{\star}^{\mathcal{I}\Gamma(as)}\left(B^{\prime} , \mathsf{R}(A,M)\right).\]
\end{cor}
\begin{proof}
The corollary follows from Theorem \ref{reflexive-thm}, upon noting that the composite 
\[\mathsf{R}(A,M) \circ i\colon \Delta R^{op}\rightarrow \mathbf{Mod}_k\]
coincides with the Loday functor $\mathcal{L}(A,M)$ of Definition \ref{loday-defn}. 
\end{proof}

\bibliographystyle{alpha}
\bibliography{PR-Type-Refs}

\end{document}